\documentclass[reqno]{amsart}

\usepackage{amsmath,amssymb,mathrsfs}
\usepackage{amsfonts}
\usepackage{amsthm}  

\def\a{\alpha}
\def\b{\beta}

\def\gm{\gamma}
\def\f{\varphi}

\def\s{\sigma}

\newcommand{\mr}[1]{\mathrm{#1}}

\def\C{\ensuremath{\mathbb{C}}}

\def\N{\ensuremath{\mathbb{N}}}

\def\R{\ensuremath{\mathbb{R}}}
\def\S{\ensuremath{\mathbb{S}}}

\def\Nd{ {\N^d}}
\def\Rd{ {\R^d}}

\def\Cd{ {\C^d}}

\def\pd{\partial}
\def\qq{\forall\,}

\newcommand{\intrd}{\int_{\Rd}}

\newcommand{\cC}{\ensuremath{\mathcal{C}}}

\newcommand{\cF}{\ensuremath{\mathcal{F}}}

\newcommand{\cP}{\ensuremath{\mathcal{P}}}

\newcommand{\abs}[1]{\left|{#1}\right|}
\newcommand{\nrm}[1]{\left\|{#1}\right\|}
\newcommand{\set}[1]{\left\{#1\right\}\relax}
\newcommand{\scal}[1]{\left\langle\relax #1 \relax\right\rangle}

\newcommand{\qtxtq}[1]{\quad \text{#1}\quad}

\newtheorem{thm}{Theorem}[section]
\newtheorem{prop}[thm]{Proposition}
\newtheorem{lem}[thm]{Lemma}
\newtheorem{cor}[thm]{Corollary}
\newtheorem{rem}[thm]{Remark}

\title{On the Dunkl Intertwining Operator}

\author{Mostafa Maslouhi}
\address{M. Maslouhi: 
Département Informatique, Logistique et Mathématiques (ILM)\\
Ecole Nationale des Sciences Appliqu\'ees\\
Universit\'e Ibn Tofail, 14000 Kenitra, Maroc.}
\email {mostafa.maslouhi@univ-ibntofail.ac.ma}
\email {mostafa.maslouhi@gmail.com}
\subjclass[2000]{47B48; 33C52; 33C67; 31B05}

\keywords{Dunkl operators, Dunkl kernel, Dunkl intertwining operator, Root systems, Weyl group.}

\begin{document}
\begin{abstract} 	
Dunkl operators are differential-difference operators parametrized by a finite reflection group and a weight function. The commutative algebra generated by these operators generalizes the algebra of  standard differential operators and intertwines with this latter by the so-called intertwining operator. In this paper, we give an integral representation for the operator $V_k\circ e^{\Delta/2}$ for an arbitrary Weyl group and a large class of regular weights  $k$ containing those of non negative real parts. Our representing measures are absolute continuous with respect the Lebesgue measure in $\mathbb{R}^d$, which allows us to derive out new results about the intertwining operator $V_k$ and the Dunkl kernel $E_k$. We show in particular that the operator $V_k\circ e^{\Delta/2}$ extends uniquely as a bounded operator to a large class of functions  which are not necessarily differentiables. In the case of non negative weights, this operator is shown to be positivity-preserving. 
\end{abstract}

\maketitle

\markboth{M. Maslouhi}{On the Dunkl Intertwining Operator}

\section{Introduction and preliminaries} \label{sec: prelim}
Dunkl operators are first-order differential-difference operators associated to a finite reflection group and a weight function. They form a commutative algebra which generalizes the algebra of standard differential operators. Nowadays, there is a harmonic analysis associated to these operators which extends the euclidean Fourier analysis by using the action of a Weyl group generated by some root systems in a finite dimensional vector space. During the last years, this theory has attracted considerable interests in various fields of mathematics\cite{DX,ROS-VOIT} and also in physics\cite{LV}.

In this theory, the so-called Dunkl kernel, plays the primordial role of the $``$exponential$"$  and is obtained as the image of the standard exponential by another key operator referred to as the intertwining operator, named after its action of intertwining the Dunkl differential operators algebra and the algebra of classical differential operators.

One challenging problem in the Dunkl operators theory is to find a representing measures for the action of the intertwining operator on the space of polynomials. This problem is solved by R\"osler\cite{ROS} in the case of non negative weights and the general case remains so far an open problem.

In this paper we reduce the gap and get closer to these representing measures. We give a family of  representing measures of the action of the operator $V_k\circ e^{\Delta/2}$ on the space of polynomials for an arbitrary Weyl group and a large class of regular weights  $k$ containing those of non negative real parts. Our representing measures are absolute continuous with respect to the Lebesgue measure in $\Rd$, and reveal new results about the intertwining operator $V_k$ and the Dunkl kernel $E_k$.

We show that the operator $V_k\circ e^{\Delta/2}$ extends uniquely as a bounded operator to a large class of functions, including polynomials, which are not necessarily differentiables. In the case when $k$ is non negative, this operator is shown to be positive-preserving. 	

This paper contains two main results and is organized as follows. The current section serves to introduce the Dunkl operators and the common notations used throughout this paper. In section \ref{sec:preparatory_calculations} we gather all technical results requested by the developments in section \ref{sec:integral_repr}  where our main results are presented. The first one gives an integral representation for the operator $V_k\circ e^{\Delta/2}$. The second one extends the operator $V_k\circ e^{\Delta/2}$ as bounded operator to a large class of functions containing polynomials, and show its positivity-preserving property when the weight $k$ is non negative.
 
\bigskip
 
We consider the euclidean space $\Rd$ equipped with its canonical inner product  $\scal{\,,\,}$ which we extend as a bilinear form on $\Cd\times\Cd$ again denoted by $\scal{\,,\,}$. For $z=(z_1,\dots,z_d)\in\Cd$, we let $$\nrm{z\,}^2:=z_1^2+\dots+z_d^2.$$ 
 
Fix a root system  $R$ and consider the finite group $G$ generated by the reflections $\s_\a$ where $\a\in R$ and $$\s_\a(x)=x-2\scal{x,\a}\a/\nrm{\a}^2,\quad (x\in\Rd).$$

Let $\cP:=\C[\Rd]$ denotes the $\C$-algebra of polynomial functions on $\Rd$ and $\cP_n$, $n\in\N$, its subspace consisting of all homogeneous polynomials of degree $n$, and for $p\in\cP_n$, for some $n\ge 0$, we set $$\nrm{p}_{\S}:=\sup_{\nrm{x}=1}\abs{p(x)}.$$

The action of $G$ on functions is defined by $$(L_g\cdot f)(x):= f(g\cdot x),\quad x\in\Rd$$ 

Set $R_+:=\{\a\in R:\,\scal{\a,\b}>0\}$  for some $\b\in\Rd$ such that $\scal{\a,\b}\ne
0$ for all $\a\in R$.

Let $k$ be a parameter function on $R$, that is,  $k:R\to \C$ and $G$-invariant.

 For $\xi\in\R^n$, the Dunkl operator $T_\xi$ on $\Rd$ associated to the group $G$ and the
 parameter function $k$ is given by
$$  T_\xi(k) f(x)=\pd_\xi f(x)+\sum_{\a\in R_+}k(\a)\scal{\a,\xi}\frac{f(x)-f(\s_\a x)}{\scal{\a,x}},\quad x\in\Rd$$ 
where $f\in C^1(\Rd)$.  In the sequel we will write $T_{j}$ in place of
$T_{e_j}(k)$, where $e_j$ is a vector from the standard basis of $\Rd$, $j=1,\dots,d$.

Consider the   vector space $M$ of all parameter functions and let $$M^{reg}:=\set{k\in M:\,
\cap_{\xi \in \Rd}\ker(T_\xi(k))=\C\cdot 1}$$ 
be  the set of regular parameter functions. 

An important result in Dunkl theory, established in \cite{DJO},  states that $k\in M^{reg}$ if and only if 
there exists a unique isomorphism $V_k$ of $\cP$ satisfying
\begin{equation}\label{eq:def principal of V_k}
    V_k(\cP_n)\subset \cP_n,\quad  V_k(1)=1\qtxtq{and} T_\xi V_k=V_k\pd_\xi, \quad
    \qq\xi\in\Rd.
\end{equation}
The operator  $V_k$ is called the intertwining operator.

In \cite{M-Y-Vk} a construction of the intertwining operator was given for a large class of regular weights. We recall here some notations and some key results from \cite{M-Y-Vk}. 

Consider the operator $A:=A_k$ defined on functions  $f: \Rd \to \C$ by
\begin{equation}\label{eq:def of A}
A(f):=\sum_{\a\in R_+}k(\a)\, L_{\s_\a}\cdot f.
\end{equation}
It is clear  that for all integer $n\ge 0$, the space $\cP_n$ is invariant under the action
 of $A$.  
 For $n\ge 1$, we  set $A_n:=A_{|\cP_n}$, we define
 \begin{equation}\label{eq:def_of_gm}
 	\gm=\gm(k):=\sum_{\a\in R_+}k(\a),
 \end{equation} 
  and we  consider the operator $W_{n}:=W_{n,k}$ defined in $\cP_{n}$ by $$W_{n}=(n+\gm)- A_n.$$ A straightforward calculations show that $W_n$ is a Euler-type operator satisfying 
$$W_n(p)(x)=\sum_{j=1}^d x_jT_j(p)(x),$$ 
 for all polynomial $p\in\cP_n$ and all $x\in\Rd$.

We denote by $M^*$ the set of all parameter functions $k$ for which $W_n$ is invertible for all $n\ge 1$. As mentioned in \cite{M-Y-Vk} (see also \cite[Corollary 2.3] {M-Y-Corr-Vk}), $M^*
\subset M^{reg}$, and contains a large class of parameter functions including those of non negative real part. For $k\in M^*$, we set $$H_n=H_{n,k}:=((n+\gm)-A_n)^{-1}.$$ 

One key result from \cite{M-Y-Vk} needed in this paper is the following.
\begin{thm}[\cite{M-Y-Vk}]\label{thm:main_exp_Vk} Assume that $k\in M^{*}$. Then there exists a sequence of functions $\lambda_{n}:=\lambda_{n,k}: G\to \C$, $n\ge 1$, such that 
	for all $n\ge 1$ we have
	$H_n=\sum_{g\in G}\lambda_{n}(g)L_{g}$ and 
	$$V_k(p)(x)=\sum_{g_1,\dots,g_n\in G}\left(\prod_{i=1}^n\lambda_n(g_i)\right)\pd_{g_1\dots g_n x}\dots \pd_{g_{n-1}g_n x}\pd_{g_n x}p$$ for all $p\in\cP_n$ and $x\in\Rd$. 	Moreover, there exists a positive constant $\delta:=\delta(k)$ such that
	$$\abs{\lambda_n(g)}\le \frac{\delta}{n}$$ for all $n\ge 1$ and all $g\in G$.
\end{thm}

From now on, we assume that $k\in M^*$, unless otherwise mentioned, and we refer to $\lambda_n$, $n\ge 1,$ and the constant $\delta$ as they were introduced in Theorem \ref{thm:main_exp_Vk}.

\begin{rem}\label{rem:V_k_real_valued} When $k$ is real valued, then $W_n$ and $H_n$ lie together in the algebra $\set{\sum_{g\in G}c(g)L_g,\, c(g)\in \R}$. As a consequence, by use of Theorem \ref{thm:main_exp_Vk}  we see that $V_k$ is real valued on polynomials $p\in\R[\Rd]$.
\end{rem}

The next useful estimate for $V_k$, valid in the case $k\in M^{*}$, is derived out from Theorem \ref{thm:main_exp_Vk}:
\begin{prop}\label{pro:estimates_Vk_on_Pn}
Let $n\ge 1$. Then for all $p\in\cP_n$ and $x\in\Rd$ we have
$$	\abs{V_k(p)(x)}\le \frac{(\delta\abs{G}\nrm{x})^n}{n!}\nrm{p}_{\S}.$$
\end{prop}

The Dunkl kernel $(x,y)\mapsto E_k(x,y)$, is a fundamental tool in the harmonic analysis associated to Dunkl operator theory. Indeed, it plays the role of a generalized exponential function and it is used to define the Dunkl Fourier transform. 

For all $y\in\Rd$, the function $f: x\mapsto E_k(x,y)$ is the unique solution of the system  $$f(0)=1,\quad (T_\xi f)(x)=\scal{\xi,y}f(x)$$ for all $\xi\in\Rd$. 
Moreover, the map $(k,x,y)\mapsto E_k(x,y)$ is holomorphic in the set $M^{\text{reg}}\times \Cd\times\Cd$.

We list here some of its basic properties needed for the sequel. For more details on this kernel see \cite{DUNKL1,dJ1} and the references there in.
\begin{prop}\label{pro:E_k_popties}\textnormal{(\cite{dJ1})}\quad
Let  $x,y\in \Cd$, $\lambda\in \C$ and $g\in G$. Then
\begin{enumerate}
    \item $E_k(x,0)=1$,
    \item $E_k(x,y)=E_k(y,x)$,
    \item $E_k(\lambda\,x,y)=E_k(x,\lambda\, y)$,
    \item $E_k(g x, y)=E_k(x,g^{-1}y)$.
\end{enumerate}
\end{prop}

In \cite{M-Y-Vk}, the kernel $E_k$ is obtained in the case $k\in M^{\ast}$ by $$E_k(x,y)=V_k\left(e^{\scal{\cdot,y}}\right)(x)=\sum_{n=0}^{\infty}E_n(x,y),\quad x\in\Rd,\ y\in \Cd$$ where $E_n(x,y):=\frac{1}{n!} V_k(\scal{\cdot,y}^n)(x)$.   

A direct application of Theorem \ref{thm:main_exp_Vk} gives us
	\begin{equation}\label{eq:exp_of_E_n}
	E_n(x,y)= \sum_{g_1,\dots,g_n\in G} \left(\prod_{i=1}^n\lambda_{n}(g_i)\right)\scal{g_1\dots g_n x,y}\dots\scal{g_nx,y}
\end{equation}
for all $n\ge 1$ and all $x,y\in\Cd$.

\section{Preparatory Results} 
\label{sec:preparatory_calculations}
The main task of this section is to prepare the ground for the results in section \ref{sec:integral_repr}.

We begin by recalling a useful identity for $n-$linear symmetric forms in $\Rd$.  
Let $\phi$ be an $n-$linear form in $\Rd$. Then following \cite[Theorem 9]{LAW} we have:
\begin{equation}\label{eq:nrm_n_linear_form}
	\sup_{\nrm{z}=1}\abs{\phi(z,\dots,z)}=\sup_{\nrm{z_1}=1,\dots,\nrm{z_n}=1}\abs{\phi(z_1,\dots,z_n)}
\end{equation}

The next technical Lemma is the corner stone for our paper. Together with
\eqref{eq:exp_of_E_n},  lead to a useful estimates for the polynomials $E_n(x,y)$ interesting for themselves and required for our developments in section \ref{sec:integral_repr}.
\begin{lem}\label{lem:power_Laplacian_estimates}
Let $n$ be a positive integer and $p$ the polynomial given by $$p(y)=\scal{\xi_1,y}\times\dots\times\scal{\xi_n,y}$$ where $\xi_i\in\Rd$ such that 
$\nrm{\xi_1}=\dots=\nrm{\xi_n}=c$ for all  $i=1,\dots,n$. Then for all $y\in \Rd$ and $0\le m\le n/2$, we have
$$\abs{\Delta^{m}p(y)}\leq \frac{n!c^n d^m}{(n-2m)!} \nrm{y}^{n-2m}$$ where $\Delta$ denotes the (standard) laplacian and the differentiation is taken with respect to $y$.
\end{lem}
\begin{proof}
Consider the n-linear symmetric form $\phi$ defined by $p(y)=\phi(y,\dots,y)$.
Letting $(e_1,\dots,e_d)$ denotes the standard basis of $\Rd$, direct calculations give us
$$\Delta p(y)=\Delta\phi(y,\dots,y)=n(n-1)\sum_{i_1=1}^d\phi(e_{i_1},e_{i_1},y,\dots,y),$$
and by induction we get 
\begin{equation}\label{eq:induction_Delta}
	\Delta^{m}p(y)=\frac{n!}{(n-2m)!}\sum_{i_1=1}^d\dots\sum_{i_m=1}^d\phi(e_{i_1},e_{i_1},\dots,e_{i_m},e_{i_m},y,\dots,y),
\end{equation}
whenever $2m\le n$.
Keeping \eqref{eq:nrm_n_linear_form} in mind, we get
\begin{align*}
\abs{\phi(e_{i_1},e_{i_1},\dots,e_{i_m},e_{i_m},y,\dots,y)}&\leq \nrm{y}^{n-2m}\sup_{\nrm{z_1}=1,\dots,\nrm{z_d}=1}\abs{\phi(z_1,\dots,z_d)}\\
& \leq \nrm{y}^{n-2m}\sup_{\nrm{z}=1}\abs{p(z)}\\
	& \leq c^n\nrm{y}^{n-2m}.
\end{align*}
Coming back to \eqref{eq:induction_Delta}, we infer that
\begin{align*}
	\abs{\Delta^{m}p(y)}\leq \frac{n!d^m}{(n-2m)!}c^n \nrm{y}^{n-2m},
\end{align*}
as claimed.
\end{proof}

To indicate the relevant variable, we will sometimes use the notations $\Delta^{m,y}E_n(x,y),$   $e^{-\Delta/2,y}E_n(x,y)$  or $\left(\Delta^{m,\cdot} E_n(x,\cdot)\right)(y)$ and  $\left(e^{-\Delta/2,\cdot} E_n(x,\cdot)\right)(y)$. Sometimes, for the sake of simplicity, we will write only $\Delta^{m}E_n(x,y)$ and  $e^{-\Delta/2} E_n(x,y)$ when the relevant variable of the differentiation is known. 

The next proposition is the fundamental tool for our developments in section \ref{sec:integral_repr}.
\begin{prop}\label{pro:En_estimates}
	\begin{enumerate}
		\item\label{itm:estimates_on_En_1} For all nonnegative integer $n$ and $0\le m\le n/2$, we have $$\abs{\left(\Delta^{m,\cdot} E_n(x,\cdot)\right)(y)}\le \frac{d^m}{(n-2m)!}(\delta\abs{G}\nrm{x})^n \nrm{y}^{n-2m} $$ for all $x,y\in\Rd$.
		\item The series $\sum_{n=0}^{\infty} \left(e^{-\Delta/2,\cdot} E_n(x,\cdot)\right)(y)$ converges uniformly with respect to $x$ and $y$ in each bounded subset 
of $\Rd\times\Rd$. Moreover, $$\sum_{n=0}^{\infty}\abs{\left(e^{-\Delta/2,\cdot} E_n(x,\cdot)\right)(y)}\le e^{(\delta \sqrt{d}\abs{G}\nrm{x})^{2}/2}e^{\delta \abs{G}\nrm{x}\nrm{y}}.$$
	\end{enumerate}
\end{prop}
\begin{proof} Fix $x\in\Rd$. 
	\begin{enumerate}
		\item Lemma  \ref{lem:power_Laplacian_estimates} applied to each polynomial 
	$$Q:y\mapsto \scal{g_1\dots g_n x,y}\scal{g_2\dots g_n x,y}\dots\scal{g_nx,y},$$ where $x\in\Rd$ and $g_1,\dots ,g_n\in G$ fixed, gives us
	\begin{equation}\label{eq:Delta_m_Q}
\abs{\Delta^{m}Q(y)}\leq \frac{n!d^m}{(n-2m)!}\nrm{x}^n \nrm{y}^{n-2m}		
	\end{equation}
Appealing to Theorem \ref{thm:main_exp_Vk} we have 
\begin{equation}\label{eq:estim_prod_Lambda_n}
	\abs{\prod_{i=1}^n\lambda_{n}(g_i)}\le \frac{\delta^n}{n!}
\end{equation}
for all $g_1,\dots,g_n\in G$. Using the expression \eqref{eq:exp_of_E_n} of $E_n$ together with \eqref{eq:Delta_m_Q} and \eqref{eq:estim_prod_Lambda_n} prove our first assertion.

		\item Using the first assertion, we have \begin{align*}
			\abs{e^{-\Delta/2,y} E_n(x,y)}&\le \sum_{0\le m\le n/2}\frac{1}{2^m m!}\abs{\Delta^{m,y} E_n(x,y)}\\
			&\le \sum_{0\le m\le n/2}\frac{d^m}{2^m m!}\frac{\nrm{y}^{n-2m}}{(n-2m)!}(\delta\abs{G}\nrm{x})^n. 
		\end{align*}
By consequence, 
\begin{align*}
&\sum_{n=0}^{\infty}\abs{e^{-\Delta/2,y} E_n(x,y)}\le\sum_{n=0}^{\infty}(\delta\abs{G}\nrm{x})^n \sum_{0\le m\le n/2}\frac{d^m}{2^m m!}\frac{\nrm{y}^{n-2m}}{(n-2m)!}\\	
&\le\sum_{m=0}^{\infty} \frac{d^m}{2^m m!}\sum_{n=2m}^{\infty}(\delta\abs{G}\nrm{x})^n\frac{\nrm{y}^{n-2m}}{(n-2m)!}\\
&\le\left(\sum_{m=0}^{\infty} \frac{(\delta \sqrt{d}\abs{G}\nrm{x})^{2m}}{2^m m!}\right)\left(\sum_{n=0}^{\infty}(\delta\abs{G}\nrm{x})^n\frac{\nrm{y}^{n}}{n!}\right).
		\end{align*}
This shows that the series converges uniformly with respect to $x$ and $y$ in each bounded subset of $\Rd\times\Rd$ and $$\sum_{n=0}^{\infty}\abs{e^{-\Delta/2,y} E_n(x,y)}\le e^{(\delta \sqrt{d}\abs{G}\nrm{x})^{2}/2}e^{\delta \abs{G}\nrm{x}\nrm{y}}$$ which proves our second assertion and ends the proof.
	\end{enumerate}
		
\end{proof}


\section{An integral representation for the operator $V_k\circ e^{\Delta/2}$} 
\label{sec:integral_repr}

Based on the results of the previous section, we will establish an integral representation for the  operator $p\mapsto \left(V_k(e^{\Delta/2}p)\right)(x)$ on the space of polynomials by means of a family of measures $\mu_x$, $x\in\Rd$, which are absolutely continuous with respect to the Lebesgue measures in $\Rd$. 
These measures have the form $$\mu_x(dz):=e^{-\nrm{z}^2/2}L_k(x,z)dz,$$ where $L_k$ is a given kernel in $\Rd\times\Rd$. 
This section is mainly devoted to the study of the measures $\mu_x$ and the kernel $L_k$.

To do so, we need first to introduce some additional specific notations for this section. 
From now on, $d\gm$ stands for the measure in $\Rd$ given by 
$$d\gm:=c_0^{-1}e^{-\nrm{z}^2/2}dz,$$ where $dz$ is the Lebesgue measure in $\Rd$ and  $c_0=(2\pi)^{d/2}$.

We let $L^p(\Rd)$ be the set of (classes of) functions defined in $\Rd$ such that
$$\nrm{f}
_{L^p(\Rd)}^p:=\intrd 
\abs{f(z)}^p dz<+\infty,$$
and $L^2(\Rd,d\gm)$ will consists  of (classes of) of functions $f$ defined in $\Rd$ and satisfying $$\nrm{f}_{2,\gm}^2:=\intrd \abs{f(z)}^2 d\gm(z)<+\infty.$$
For $f\in L^1(\Rd)$, the Fourier transform of $f$ will be defined by $$\cF(f)(y):=c_0^{-1}\intrd f(z)e^{-i\scal{y,z}}dz.$$

Finally, using the standard multi-index notation in $\Nd$, we introduce the polynomials $\f_\nu$, $\nu\in\Nd$, by:   
\begin{equation}\label{eq:def_fi_nu}
\f_\nu(X)=\frac{1}{(\nu !)^{1/2}} X^{\nu}.
\end{equation}

For a polynomial $p$, we denote  by $p(\pd)$ the differential operator consisting by replacing $X_j$ in $p$ by $\pd_j:=\frac{\pd}{\pd x_j}$, and we define the Fischer bilinear form $[,]$ in 
$\cP$ by 
\begin{equation}\label{eq:Fisher_inner_prod}
	[p,q]:=p(\pd)(q)(0)
\end{equation}

The next Proposition gathers some well-known properties of this bilinear form. 
\begin{prop}\label{pro: fisher prod pties} 
	\begin{enumerate}
		\item The system $(\f_\nu)_{\nu\in\Nd}$ forms an orthonormal basis of $\cP$ relatively to the Fischer bilinear form $[,]$.
\item\label{itm: 1 dot p} $[1,p]=p(0)$ for all polynomial $p$.
\item\label{itm: orthogonal} If $p\in \cP_n$ and $q\in \cP_m$ with $n\neq m$, then $[p,q]=0$,
\item\label{itm: xiP} $[x_i p,q] = [p,\pd_i q]$ for all $p,q\in \Pi$ and $i = 1,\dots,d$.
\end{enumerate}
\end{prop}

Macdonald \cite{MacD} states that
\begin{equation}\label{eq:fisher prod expr}
[p,q]=\intrd \left(e^{-\Delta/2}p\right)(z)\left(e^{-\Delta/2}q\right)(z)d\gm	
\end{equation}
for all polynomials $p,q$. 

For $\nu\in\N^d$, we introduce the Hermite polynomial $H_\nu$ and its associated Hermite function $h_{\nu}$ by:
\begin{equation}\label{eq:def_H_nu}
H_{\nu}(z):=\left(e^{-\Delta/2}\f_{\nu}\right)(z),\quad h_{\nu}:=e^{-\nrm{\cdot}^2/2}H_{\nu}.	
\end{equation}
The systems $(H_\nu)_{\nu\in\Nd}$ and $(h_\nu)_{\nu\in\Nd}$ form an orthonormal bases of $L^2(\Rd,d\gm)$ and $L^2(\Rd,c_0^{-1}dz)$ respectively, equipped with their standard inner product.

By use of Proposition \ref{pro: fisher prod pties}.\eqref{itm: xiP},
and the multinomial formula $$\frac{1}{n!}(X_1 + X_2 + \cdots + X_d)^n
 = \sum_{\mu\in\Nd,\abs{\mu}=n} \frac{X^\mu}{\mu!},$$
 one has for all $\nu\in\Nd$, with $\abs{\nu}=n$, the useful identity
\begin{equation}\label{eq:p_fi_mu}
	  \Big[\scal{x,\cdot}^n ,\f_{\nu}\Big] =n!\f_\nu(x),
	 \end{equation}
for all $x\in\Rd$, where the pairing $\scal{\cdot,\cdot}$ is the standard inner product in $\Rd$. 

The starting point of our integral representation is the next Proposition.
\begin{prop}\label{pro: Vk by means En} Let $p$ be a polynomial. Then for all $x\in\Rd$
we have:	$$V_k(p)(x)=\sum_{n=0}^{\infty}[E_n(x,\cdot),p].$$	
\end{prop}
\begin{proof} Fix an integer $n\ge 0$. By \eqref{eq:p_fi_mu}, we see that 
	\begin{equation}\label{eq:scal_x_y_n_by_fi_mu}
		\frac{1}{n!}\scal{x,y}^n=\sum_{\nu\in\Nd,\abs{\nu}=n}\f_\nu(x)\f_\nu(y),
	\end{equation}
for all $x,y\in\Rd$. Applying $V_k$, with respect to $x$, to both sides of \eqref{eq:scal_x_y_n_by_fi_mu} leads to
$$E_n(x,y)=\sum_{\nu\in\Nd,\abs{\nu}=n}V_k(\f_\nu)(x)\f_\nu(y),$$
which in turn gives 
\begin{equation}\label{eq:V_k_as_coefficient}
	V_k(\f_\nu)(x)=[E_n(x,.),\f_\nu]
\end{equation}
for all $\nu\in\Nd$ with $\abs{\nu}=n$. 

Fix $\nu\in\Nd$. By virtue of  \eqref{eq:V_k_as_coefficient} and the orthogonality property of the Fisher pairing  $[\,,\,]$ ensured by Proposition \ref{pro: fisher prod pties}.\eqref{itm: orthogonal}, we may write 
   $$V_k(\f_\nu)(x)=\sum_{m=0}^{\infty}[E_m(x,.),\f_\nu].$$ 
Since $(\f_\nu)_{\nu\in\N}$ is a basis of $\cP$ and $V_k$ is linear, the latter formula holds for all polynomial $p$, and this proves our claim.
\end{proof}

Our first main result is the following.
\begin{thm}\label{thm:1st_integral_form}
There exists a unique kernel $L_k(\cdot,\cdot)$ on $\Rd\times\Rd$ satisfying $$V_k(p)(x)=\intrd L_k(x,y)e^{-\Delta/2}p(y)d\gm(y),$$ for all polynomial $p \in\cP$ and all $x\in\Rd$. Moreover,  
\begin{equation}\label{eq:exp of Lk}
	L_k(x,y)=\sum_{n=0}^\infty \left(e^{-\Delta/2,\cdot}E_n(x,\cdot)\right)(y),
\end{equation}
and satisfies
\begin{equation}\label{eq:main_estimates_Lk}
	\abs{L_k(x,y)}\le e^{(\delta \sqrt{d}\abs{G}\nrm{x})^{2}/2}e^{(\delta \abs{G}\nrm{x}\nrm{y})/2}
\end{equation}
for all $x,y\in\Rd$.
\end{thm}
\begin{proof} 	
Pick $x\in\Rd$ and fix $p\in\cP$. From  Proposition \ref{pro: Vk by means En} we have $$	V_k(p)(x)=\sum_{n=0}^\infty 
\intrd \left(e^{-\Delta/2,\cdot}E_n(x,\cdot)\right)(y) e^{-\Delta/2}p(y) d\gm(y).$$ The dominate convergence theorem, thanks to Proposition \ref{pro:En_estimates}, yields that 	
$$V_k(p)(x)=\intrd L_k(x,y) e^{-\Delta/2}p(y) d\gm(y),$$ 
where the kernel $L_k$ is given by: $$L_k(x,y):=\sum_{n=0}^\infty (e^{-\Delta/2,\cdot}E_n(x,\cdot))(y).$$	
 The unicity of $L_k$ follows from the density of  $\cP$ in $L^2(\Rd,d\gm)$.
 The remaining is a direct application of Proposition \ref{pro:En_estimates}, which completes the proof.\end{proof}

We begin by giving some direct properties of the kernel $L_k$. Others properties of $L_k$ will be given bellow in their appropriate places.
\begin{prop}\label{pro: L_k properties}Fix $x\in\Rd$. We have:
\begin{enumerate}
	\item $\intrd L_k(x,y)d\gm(y)=1$,
	\item $L_k$ is continuous in $\Rd\times\Rd$,
	\item  $L_k$ is real valued in $\Rd\times\Rd$ whenever $k$ is real valued.
	\item For  $y\in\Rd$ fixed, the map
	$x\mapsto L_k(x,y)$ extends to a holomorphic function in $\Cd$. 
\end{enumerate}
\end{prop}
\begin{proof}
	\begin{enumerate}
		\item Follows from $V_k(1)=1$.
		\item Follows from Proposition \ref{pro:En_estimates}.
		\item Keeping remark \ref{rem:V_k_real_valued} in mind, we see that $E_n(x,\cdot)$ is  real valued when $k$ is real valued, by consequence $L_k$ is real valued in $\Rd\times\Rd$  when $k$ is real valued.
		\item For $y$ fixed, the maps $x\mapsto e^{-\Delta/2}E_n(x,y)$ are analytic in $\Cd$ and arguing as in Proposition \ref{pro:En_estimates}, we see that the series defining $L_k$ converges uniformly with respect to $x\in\Cd$ in each bounded subset of $\Cd$, this ensures that $x\mapsto L_k(x,y)$ is holomorphic in $\Cd$.
	\end{enumerate}
	Thus, all our claims are proved. 
\end{proof}

The next Proposition gives a generating series for the intertwining operator $V_k$.
\begin{prop}\label{pro:L_k_decomp_in_L2}
	For all $x\in \Cd$, we have
	\begin{equation}\label{eq:L_k_rxpr_in_L2_gm}
		L_k(x,\cdot)=\sum_{\nu\in\N^d}V_k(\f_{\nu})(x)H_{\nu}(\cdot),
	\end{equation}  where the above equality holds point wise in $\Rd$ and in $L^2(\Rd,d\gm)$ as well.
	Moreover, 
	\begin{equation}\label{eq:nrm_Lk_in_L2}
	\nrm{e^{-\nrm{\cdot}^2/4}L_k(x,\cdot)}^2_{L^2(\Rd)}=c_0^{-1}\sum_{\nu\in\N^d}\Big|V_k(\f_{\nu})(x)\Big|^2
	\end{equation}
\end{prop}
\begin{proof}	
Fix $x\in\Cd$. Theorem \ref{thm:1st_integral_form} gives us
\begin{align*}
V_k(\f_{\nu})(x)=\intrd L_k(x,y)(e^{-\Delta/2}\f_{\nu})(y)d\gm=\intrd L_k(x,y)H_{\nu}(y)d\gm.
\end{align*}
 Since $H_{\nu}$ is real valued and $(H_{\nu})_{\nu\in\Nd}$ is an orthonormal basis  of $L^2(\Rd,d\gm)$, we infer that
$$L_k(x,\cdot)=\sum_{\nu\in\N^d}V_k(\f_{\nu})( x)H_{\nu},\qtxtq{in} L^2(\Rd,d\gm).$$
This proves \eqref{eq:L_k_rxpr_in_L2_gm} in $L^2(\Rd,d\gm)$.

Alternatively, by view of Lemma \ref{lem:power_Laplacian_estimates} and the estimates provided by Proposition \ref{pro:estimates_Vk_on_Pn}, the series $y\mapsto\sum_{\nu\in\N^d}V_k(\f_{\nu})( x)H_{\nu}(y)$ converges point wise in $\Rd$ and uniformly in bounded subset of $\Rd$, and then \eqref{eq:L_k_rxpr_in_L2_gm} holds point wise almost every where w.r.t the lebesgue measure in $\Rd$. By a continuity argument, \eqref{eq:L_k_rxpr_in_L2_gm} holds actually every where in $\Rd$.

In other hand, \eqref{eq:L_k_rxpr_in_L2_gm} leads also to
$$e^{-\nrm{\cdot}^2/2}L_k(x,\cdot)=\sum_{\nu\in\N^d}V_k(\f_{\nu})( x)h_{\nu}(\cdot), \qtxtq{in} L^2(\R^2,c_0^{-1}dz)$$ and then taking the norm in $L^2(\Rd)$ in \eqref{eq:L_k_rxpr_in_L2_gm} we get $$	\nrm{e^{-\nrm{\cdot}^2/4}L_k(x,\cdot)}^2_{L^2(\Rd)}=c_0^{-1}\sum_{\nu\in\N^d}\Big|V_k(\f_{\nu})(x)\Big|^2,$$ which is the desired result.
\end{proof}


Another important property of the kernel $L_k$ is its relationship with the Dunkl kernel. The next theorem will show that the Dunkl kernel $E_k$ is obtained as the convolution of the kernel $L_k$ and the gaussian function $e^{-\nrm{\cdot}^2/2}$. This result will be the source of our second main result in Theorem \ref{thm:V_k_by_Lk} bellow.
\begin{thm}\label{thm:exp_integ_Ek_by_Lk}
For all $x,y\in\Cd$	we have 
$$E_k(x,y)=c_0^{-1}\intrd L_k(x,z) e^{-\nrm{y-z}^2/2}dz=c_0^{-1}L_k(x,\cdot)*e^{-\nrm{\cdot}^2/2}(y)$$ 
\end{thm}
\begin{proof}
	Let us assume first that $x,y\in\Rd$. Note that in this case, Theorem \ref{thm:1st_integral_form} asserts that the integral is well defined for all $x,y\in\Rd$. In other hand, by the definition of $L_k$ and Proposition \ref{pro:En_estimates} we have
\begin{align*}
	\intrd L_k(x,z) e^{-\nrm{z-y}^2/2}dz&=\sum_{n=0}^{\infty}\intrd \left(e^{-\Delta/2,\cdot}E_n(x,\cdot)\right)(z) e^{-\nrm{z-y}^2/2}dz\\
&=\sum_{n=0}^{\infty}\intrd \left(e^{-\Delta/2,\cdot}E_n(x,\cdot)\right)(z+y) e^{-\nrm{z}^2/2}dz\\
&=\sum_{n=0}^{\infty}\intrd \left(e^{-\Delta/2,\cdot}E_n(x,\cdot+y)\right)(z) e^{-\nrm{z}^2/2}dz\\
&=c_0\sum_{n=0}^{\infty}\left[E_n(x,\cdot+y),1\right]
\end{align*} 
From Proposition \ref{eq:Fisher_inner_prod}.\eqref{itm: 1 dot p} we have $E_n(x,y)=[E_n(x,\cdot+y),1]$, from which  the result for $x,y\in\Rd$ follows. The extension to $x,y\in\Cd$ is obtained by analytic continuation by view of the estimates on the kernel $L_k$ furnished by Theorem \ref{thm:1st_integral_form}.
\end{proof}
A direct application is the following.
\begin{cor}\label{cor:Fourier_Lk_by_Ek}For all $x\in\Cd$, $y\in\Rd$ we have
	$$e^{-\nrm{y}^2/2}L_k(x,y)=\cF(e^{-\nrm{\cdot}^2/2}E_k(-ix,\cdot))(y)$$
\end{cor}
\begin{proof}
Noting that $E_k(ix,y)=E_k(x,iy)$, we deduce from Theorem \ref{thm:exp_integ_Ek_by_Lk} that
for all $x,y\in\Rd$ we have $$E_k(ix,y)=e^{\nrm{y}^2/2}c_0^{-1}\intrd L_k(x,z)e^{-i\scal{y,z}} e^{-\nrm{z}^2/2}dz.$$ That is, 
$$e^{-\nrm{y}^2/2}E_k(ix,y)=\cF(e^{-\nrm{\cdot}^2/2}L_k(x,\cdot))(y)$$ for all $x,y\in\Rd$, and then 
$$L_k(x,y)e^{-\nrm{y}^2/2}=\cF(e^{-\nrm{\cdot}^2/2}E_k(-ix,\cdot))(y),$$ as claimed.
\end{proof}

Proposition \ref{pro:pd_j_and_Tj} and Proposition \ref{pro:L_k_and_G_action} bellow, are direct  applications of Corollary \ref{cor:Fourier_Lk_by_Ek} which reveal further properties of the kernel $L_k$.
\begin{prop}\label{pro:pd_j_and_Tj}We have $L_k\in\cC^{\infty}(\Rd\times\Rd)$ and 
$$\pd_{y_j}\left(\left(L_k(x,\cdot)e^{-\nrm{\cdot}^2/2}\right)\right)(y)
=T_j^x \left(\left(L_k(\cdot,y)e^{-\nrm{y}^2/2}\right)\right)(x),$$	for all $x,y\in\Rd$ and all $j=1,\dots,d$.
\end{prop}

\begin{prop}\label{pro:L_k_and_G_action}For all $x,y\in\Rd$ and all $g\in G$, we have
		$$L_k(-x,y)= L_k(x,-y) \qtxtq{and} L_k(g x,y)=L_k(x,g^{-1}y),$$
\end{prop}

One more fundamental result on the kernel $L_k$. It provides an interesting relationship between the kernel $L_k$ and the intertwining operator $V_k$. This relationship reveals in particular, that the kernel $L_k$ is non negative in the case where the parameter function $k$ is non negative.
\begin{thm}\label{thm:V_k_by_Lk}
Assume that $\mr{Re}(k)\ge 0$. Then for all $x,y\in\Rd$, we have $$V_k(e^{-\nrm{\cdot+y}^2/2})(x)= e^{-\nrm{y}^2/2}L_k(x,y).$$ As a consequence, the kernel $L_k$ is non negative in $\Rd\times\Rd$ if $k$ is non negative.
\end{thm}
\begin{proof}
Fix $x\in\Rd$. Following dejeu \cite[Theorem 5.1]{dJ2}, we have
\begin{align*}
	V_k(e^{-\nrm{\cdot+y}^2/2})(x)&=c_0^{-1}\intrd \cF(e^{-\nrm{\cdot+y}^2/2})(z)E_k(ix,z)dz\\
	&= c_0^{-1}\intrd \cF(e^{-\nrm{\cdot}^2/2})(z)e^{i\scal{y,z}}E_k(ix,z)dz\\
	&=c_0^{-1}\intrd e^{-\nrm{z}^2/2}e^{i\scal{y,z}}E_k(ix,z)dz
\end{align*}
whence, $$V_k(e^{-\nrm{\cdot+y}^2/2})(x)=\cF\left(e^{-\nrm{\cdot}^2/2}E_k(-ix,\cdot)\right)(y),$$
and our first claim follows from Corollary \ref{cor:Fourier_Lk_by_Ek}.

Assume now that $k$ is non negative. Following R\"osler \cite{ROS}, $V_k(e^{-\nrm{\cdot+y}^2/2})(x)\ge~0$ for all $x,y\in\Rd$, we infer that $L_k$ is non negative in $\Rd\times\Rd$ and this ends the proof.
\end{proof}

Our second main result in this paper is about the extension of $V_k\circ e^{\Delta/2}$ as a bounded operator to a large class of functions than polynomials.  

A linear form $\Lambda$ defined on some functional space $E$ will be called positive or positivity-preserving, if we have:
$$f\in E,\  f\ge 0\implies \Lambda(f)\ge 0.$$ 
\begin{thm}\label{thm:extension_Vk_exp_Delta}
	Let $x\in\Rd$. Then the linear functional $$\Phi_{x,k}:=\Phi_x: p\mapsto V_k(e^{\Delta/2}p)(x)$$ defined on polynomials, extends uniquely to $L^2(\Rd,d\gm)$ as a bounded operator, bearing the same name, with: $$\nrm{\Phi_x}= c_0^{-1/2}\left(\sum_{\nu\in\N^d}\abs{V_k(\f_{\nu})(x)}^2\right)^{1/2}.$$ Further, in the case where $k$ is non negative, $\Phi_{x}$ is positive on  $L^2(\Rd,d\gm)$.
\end{thm}
\begin{proof}
Appealing to Theorem \ref{thm:1st_integral_form} we see, using  H\"older inequality, that  $$\abs{V_k(e^{\Delta/2}p)(x)}\le \nrm{L_k(x,\cdot)}_{2,\gm}\nrm{p}_{2,\gm}$$
for all polynomial $p$. Since polynomials are dense in $L^2(\Rd,d\gm)$, then $\Phi_x$ extends uniquely, conserving the same name, to the space $f\in L^2(\Rd,d\gm)$ by setting 
$$\Phi_x(f):=\intrd L_k(x,y)f(y)d\gm(y)$$ for all $f\in L^2(\Rd,d\gm)$. Hence
$\nrm{\Phi_x}=\nrm{L_k(x,\cdot)}_{L^2(\Rd,d\gm)}$ and by view of Proposition \ref{pro:L_k_decomp_in_L2}, our first claim follows.
	
If further $k$ is non negative, then by Theorem \ref{thm:1st_integral_form}, the kernel $L_k$ is non negative in $\Rd\times\Rd$, and this completes the proof. 
	\end{proof}

We conclude this work by noting that Theorem \ref{thm:extension_Vk_exp_Delta} allows us to define $e^{\Delta/2}(f)$ for $f\in L^2(\Rd,d\gm)$ by setting $$(e^{\Delta/2}f)(z):=V_k^{-1}\left[\left(V_k\circ e^{\Delta/2}\right)(f) \right](z),\quad z\in\Rd.$$ Indeed: The right hand side of the  above formula makes sense since by application of Theorem \ref{thm:extension_Vk_exp_Delta}, we see that for each $f\in L^2(\Rd,d\gm)$ the function $$\psi: x\mapsto \Phi_x(f)=\left(V_k\circ e^{\Delta/2}\right)(f)(x)$$ is smooth in $\Rd$, and then $V_k^{-1}(\psi)$ is well defined.


\end{document}